\documentclass[12pt]{amsart}
\usepackage{dsfont,graphicx,xcolor,amsmath, amssymb,float,bbm,todonotes,enumerate,comment}
\usepackage{algorithm}
\usepackage[noend]{algpseudocode}
\usepackage{mathtools}
\mathtoolsset{showonlyrefs,showmanualtags}

\newtheorem{theorem}{Theorem}
\newtheorem{lemma}{Lemma}
\newtheorem{prop}{Proposition}

\newtheorem{corollary}{Corollary}

\newcommand\R{{\mathbb R}}

\newcommand\Z{{\mathbb Z}}
\newcommand\N{{\mathbb N}}
\newcommand\T{{\mathbb T}}
\newcommand{\be}{\begin{equation}}
\newcommand{\ee}{\end{equation}}
\newcommand{\bea}{\begin{eqnarray}}
\newcommand{\eea}{\end{eqnarray}}
\newcommand{\barr}{\begin{array}}
\newcommand{\earr}{\end{array}}
\newcommand{\e}{\mathrm{e}}

\def\XXint#1#2#3{{\setbox0=\hbox{$#1{#2#3}{\int}$}
     \vcenter{\hbox{$#2#3$}}\kern-.5\wd0}}

\renewcommand{\Re}{\operatorname{Re}}

\title[Jumps, cusps and fractals in the solution of BO]{Jumps, cusps and fractals in the solution of the periodic linear Benjamin-Ono equation}
\author[L Boulton, B Macpherson, B Pelloni]{Lyonell Boulton, Breagh Macpherson and Beatrice Pelloni}
\address{Heriot-Watt University \& Maxwell Institute for the Mathematical Sciences}
\date{\today}

\begin{document}

\begin{abstract}
We establish two complementary results about the regularity of the solution of the periodic initial value problem for the linear Benjamin-Ono equation. We first give a new simple proof of the statement that, for a dense countable set of the time variable, the solution is a finite linear combination of copies of the initial condition and of its Hilbert transform. In particular, this implies that discontinuities in the initial condition are propagated in the solution as logarithmic cusps.  We then show that,  if the initial condition is of bounded variation (and even if it is not continuous), for almost every time the graph of the solution in space is continuous but fractal, with upper Minkowski dimension equal to $\frac32$. In order to illustrate this striking dichotomy, in the final section we include accurate numerical evaluations of the solution profile, as well as estimates of its box-counting dimension for two canonical choices of irrational time.  
\end{abstract}

\maketitle


\section{Introduction}
The phenomenon of cusp revivals for dispersive time-evolution equations describes the emergence of logarithmic cusp singularities in the solution, even for bounded initial conditions. These singularities are fully characterised in terms of the jump discontinuities of the  initial condition, and they occur for values of the time variable which are rationally related to the length of the period. 

In this note we analyse the regularity and behaviour of the solution of one such equation, the linear Benjamin-Ono (BO) equation, on the torus $\mathbb{T}=(-\pi,\pi]$, \be\label{lBO}
\begin{aligned} &\partial_t u(x,t)=\mathcal{H}\partial_x^2u(x,t), \qquad x\in\mathbb T, \quad t\in \mathbb{R}, \\ & u(x,0)=u_0(x);\end{aligned} 
\ee 
where $\mathcal{H}$ is the periodic Hilbert transform.
For this boundary-value problem, it is known that at any rational time, $t\in2\pi\mathbb{Q}$, the solution has a simple closed expression in terms of the initial condition, $u_0:\mathbb{T}\longrightarrow \mathbb{R}$ and of its Hilbert transform \cite{boulton2020new}. This implies that, at these times, initial jump discontinuities of $u_0$ generate cusp singularities in the solution $u(\cdot,t)$. To our knowledge, the regularity of the solution at other times has not been previously analysed.  

This behaviour is in stark contrast with the properties of the solution of a closely related problem, the linear Schr\"odinger equation on the torus,
\be\label{lsc}
\begin{aligned} &\partial_t v(x,t)=-i\partial_x^2v(x,t), \qquad x\in\mathbb T, \quad t\in \mathbb{R}, \\ & v(x,0)=v_0(x).\end{aligned} 
\ee
In this case, if  $v_0$ is of bounded variation, then the solution is bounded for all $t\in \mathbb{R}$, see \cite{rodnianski2000fractal}.    

Our purpose is to highlight a simple argument that characterises the regularity of the solution to \eqref{lBO} in terms of the solution to \eqref{lsc}.  This argument yields a new short, rigorous proof of the phenomenon of cusp revivals for \eqref{lBO}, for any discontinuous $u_0\in \mathrm{L}^2(\mathbb{T})$. In addition, it leads to the proof of the following striking behaviour: despite displaying cusp singularities for all $t$ in a dense subset of $\mathbb{R}$, the solution of \eqref{lBO} has the same regularity, as measured in Besov spaces, as the solution of \eqref{lsc}. In particular, it is H\"older continuous for almost every $t\in\mathbb{R}$, provided $u_0$ is of bounded variation. 

Cusp revivals are known to occur in linear integro\--\-differ\-ential equations such as the linear BO equation. This was first observed in \cite{boulton2020new}, where a formal proof of the general case was given. They are also known to occur in linear differential dispersive boundary-value problems of odd order or with dislocations, where they are induced by the boundary or dislocation conditions, \cite{bfps2024}. In all these cases, the solution at rational times is the sum of three components: a finite linear combination of translated copies of a function containing all the jump discontinuities of the initial condition, the Hilbert transform of this function and a continuous function. The Hilbert transform of the jump discontinuities is responsible for the logarithmic cusps in the solution. 

Our main contribution is the statement of the next theorem, which formulates rigorously the properties of the solution to \eqref{lBO}. In this case, the continuous function component at rational times is identically zero, and the conclusion of the theorem is analogous to the quantum Talbot effect described in \cite{berry1996quantum} for the equation \eqref{lsc}, which was examined in detail in \cite{taylor2003schrodinger,olver2010dispersive,kapitanski1999does,rodnianski2000fractal} (see also \cite{chousionis2014fractal} and \cite[Section~2.2]{erdougan2016dispersive}).

\smallskip
Here and elsewhere below, BV$(\mathbb T)$ is the space of functions of bounded variation, $\mathrm{C}^{\alpha}(\mathbb{T})$ is the space of H\"older continuous functions with exponent $\alpha\in(0,1)$ and $\mathrm{H}^r(\mathbb T)$ is the $\mathrm{L}^2$-Sobolev space with derivative order $r>0$.

\begin{theorem}    \label{revfracBO}
	Let $u_0:\mathbb{T}\longrightarrow \mathbb{R}$. The following holds true for $u$ the solution to \eqref{lBO}. 
\begin{enumerate}[(a)]
\item If $u_0\in \mathrm{L}^2(\T)$, then, for $p,\,q\in \mathbb{N}$ co-prime,
\be\label{bosum}
u\Big(x,2\pi\frac{ p}{q}\Big)=\frac{1}{q}\operatorname{Re}\sum_{k=0}^{q-1}\left[\sum_{m=0}^{q-1}\e^{2\pi i\frac{km+pm^2}{q}}\right](I+i\mathcal{H}) u_0\Big(x-2\pi\frac{k}{q}\Big).
\ee
\item If $u_0\in \mathrm{BV}(\mathbb T)$, then, for almost all $t\in \R$,
$u(\cdot, t)\in \mathrm{C}^\alpha(\T)$ for all $\alpha\in[0,\frac12)$.
\item If $u_0\not\in \mathrm{H}^{r_0}(\T)$ for some $r_0\in\big[\frac12,1\big)$, then, for almost all $t\in \R$, $u(\cdot,t)\not\in \mathrm{H}^r(\mathbb{T})$ for any $r>r_0$.
\end{enumerate}\end{theorem}

Part (a) of this theorem states that the solution, at  any time of the form  $t=2\pi p/q$, is a superposition of translated copies of the initial profile and of its Hilbert transform. This implies the validity of the cusp revival phenomenon, first described in \cite{boulton2020new}, because the periodic Hilbert transform of a jump discontinuity generates a logarithmic cusp singularity. We  discuss and illustrate this property of $\mathcal H$  in Section~\ref{s3}, see   Figure \ref{Fig1}.

The next corollary states the validity of the analogous statement for Schr\"odinger's equation \eqref{lsc}, established in \cite{rodnianski2000fractal}.

\begin{corollary}\label{lbodim}
If $u_0\in \mathrm{BV}(\T)\setminus \bigcup_{s>\frac12} \mathrm{H}^s(\mathbb{T})$, then  the upper Minkowski dimension of the graph of $u(\cdot,t)$ is equal to $\frac32$ for almost every $t\in\mathbb{R}$.
\end{corollary}

We give the proofs of Theorem~\ref{revfracBO} and Corollary~\ref{lbodim} in the next section. A crucial role  is played by the identity
\[
u(x,t)=\operatorname{Re}\left[(I+i{\mathcal H})v(x,t)\right],
\]
relating the solutions of \eqref{lBO} and \eqref{lsc} that start from the same initial condition,  $v_0=u_0$. 
This identity  is a consequence of the fact  that the action of $\mathcal{H}$ preserves the
eigenfunctions of the differential operator.

In the final section we give an illustration, by means of the numerical approximations of canonical examples,  of the significance of the statement (a) and of the corollary.  


\section{Proof of the main results}
The periodic Hilbert transform $\mathcal H:\mathrm{L}^2(\mathbb{T})\longrightarrow \mathrm{L}^2(\mathbb{T})$ is the singular integral operator defined by the principal value
\be\label{HilbOp}
   \mathcal{H}u(x)=\frac{1}{2\pi}\ \operatorname{p.v.}\!\!\int_{-\pi}^{\pi} \operatorname{cot} \frac{x-y}{2} u(y)\,\mathrm{d}y.
\ee
Let $e_n(x)=\frac{1}{\sqrt{2\pi}}\mathrm{e}^{i n x}$. Then, $\{e_n\}_{n\in\mathbb{Z}}\subset \mathrm{L}^2(\mathbb{T})$ is an orthonormal basis of eigenfunctions for both $\mathcal{H}$ and the Laplacian, $-\partial^2:\mathrm{H}^2(\mathbb{T})\longrightarrow \mathrm{L}^2(\mathbb{T})$. 
Indeed, for all $n\in\mathbb{Z}$,
\be\label{Hsp}
 -\partial^2 e_n=n^2e_n;\qquad   \mathcal{H} e_n=-i\operatorname{sgn}(n) e_n.
\ee
Here and everywhere below, we write the Fourier coefficients of $f\in\mathrm{L}^2(\mathbb{T})$, with one of the usual scalings on $\mathbb{T}=(-\pi,\pi]$, as
\[\hat{f}(n)=\frac 1 {\sqrt{2\pi}}\langle f,e_n\rangle= \frac {1}{2\pi }\int_{-\pi}^{\pi} \e^{-iny}f(y)\,\mathrm{d}y.\]
Thus, we have
\begin{align*}
    &
{\mathcal H}u(x)=i\sum_{n=1}^\infty[ \hat u(-n)\e^{-inx}-\hat u(n)\e^{inx}],
\\
&{\mathcal H}\partial^2u(x)=i\sum_{n=1}^\infty n^2[\hat u(n)\e^{inx}-\hat u(-n)\e^{-inx}].
\end{align*}

This implies that, for any $u_0\in \mathrm{L}^2(\mathbb{T})$, the solution to \eqref{lBO} is given by the expression
\begin{align}\label{solution}
u(x,t)&=\sum_{n=-\infty}^\infty\e^{inx}\e^{in|n|t}\widehat{u_0}(n)\nonumber\\
&= \widehat{u_0}(0)+\sum_{n=1}^\infty[\e^{inx}\e^{in^2t}\widehat{u_0}(n)+\e^{-inx}\e^{-in^2t}\widehat{u_0}(-n)].
\end{align}
Since $u_0$ is real-valued, $\widehat{u_0}(-n)=\overline{\widehat{u_0}(n)}$ and so
$$
\e^{-inx}\e^{-in^2t}\widehat{u_0}(-n)=\overline{\e^{inx}\e^{in^2t}\widehat{u_0}(n)}.
$$
Hence, $u$ is also real-valued and
\begin{equation} \label{soltoBO}
u(x,t)=\widehat{u_0}(0)+2\Re\left[\sum_{n=1}^\infty\e^{inx}\e^{in^2t}\widehat{u_0}(n)\right].
\end{equation}
This representation provides the link between the solutions of \eqref{lBO} and \eqref{lsc}, as stated in the next proposition. 

\begin{prop}\label{uvmainrel}
Let $u_0\in \mathrm{L}^2(\T)$ be real-valued and let $v_0=u_0$. Then, the solutions to \eqref{lBO} and \eqref{lsc} are such that,
 \be\label{uvrel}
u(x,t)=\operatorname{Re}\left[(I+i\mathcal{H})v(x,t)\right].
\ee   
\end{prop}
\begin{proof}
Since $\overline{\widehat{u_0}(n)}=\widehat{u_0}(-n)$, the solution to \eqref{lsc} with $v_0=u_0$ is given by    
\[
 v(x,t)= \widehat{u_0}(0)+\sum_{n=1}^\infty [\e^{inx}\e^{in^2t}\widehat{u_0}(n)+\e^{-inx}\e^{in^2t}\overline{\widehat{u_0}(n)}].
\]
Then, using \eqref{Hsp}, we have
$$
\sum_{n=1}^\infty\e^{inx}\e^{in^2t}\widehat{u_0}(n)=\frac{v(x,t)-\widehat{u_0}(0)+i{\mathcal H}v(x,t)}{2}.$$
Replacing the sum of this expression and of its conjugate into equation \eqref{soltoBO} gives the relation \eqref{uvrel}.\end{proof}

For later purposes, note that the expression \eqref{uvrel} can be formulated in operator form as
\begin{equation}
 \mathrm{e}^{\mathcal{H}\partial^2 t}u_0=\widehat{u_0}(0)+2\operatorname{Re}\left(\mathrm{e}^{-i\partial^2 t}\Pi u_0- \widehat{u_0}(0)\right)=2\operatorname{Re}\left(\Pi\mathrm{e}^{-i\partial^2 t} u_0\right)- \widehat{u_0}(0),
  \label{withSzego}
 \end{equation}
where $\Pi f=\sum_{n=0}^\infty \hat{f}(n)\mathrm{e}^{inx}$ is the Szeg\"o projector. Indeed, the latter commutes with both $\mathcal{H}$ and  $-\partial^2$. 

\medskip

Proposition~\ref{uvmainrel} is the main ingredient in the proof of Theorem~\ref{revfracBO}, combined with the analogous statements for \eqref{lsc}, which are given in a series of papers (see e.g. \cite{kapitanski1999does,rodnianski2000fractal,chousionis2014fractal}). 

\medskip

We consider the proof of the three statement (a)-(c) separately. The proof of (a) and the proof of (c) follow immediately from \eqref{uvrel}.

\begin{proof}[Proof of Theorem~\ref{revfracBO}-(a)]
Let $v$ be the solution to \eqref{lsc} with $v_0=u_0\in\mathrm{L}^2(\T)$. Then, as shown e.g. in \cite{erdougan2016dispersive,taylor2003schrodinger}, for any $p,\,q\in\mathbb{Z}$ co-prime, the solution has the following representation at rational times $t=2\pi\frac pq$:
\[
     v\Big(x,2\pi\frac{p}{q}\Big)= \frac{1}{q}
      \sum_{k,m=0}^{q-1} \mathrm{e}^{2\pi i \frac{km}{q}}\e^{2\pi i \frac{p}{q}m^2} u_0\Big(x-2\pi\frac{k}{q}\Big).
\]
Substitution into \eqref{uvrel} gives \eqref{bosum}.
\end{proof}

\begin{proof}[Proof of Theorem~\ref{revfracBO}-(c)]
By hypothesis, $u_0\not \in \mathrm{H}^{r_0}(\mathbb{T})$ for some $r_0\in[\frac12,1)$. Since $u_0$ is real-valued, then $u_0=\Pi u_0+\overline{\Pi u_0}-\widehat{u_0}(0)$. Hence, $\Pi u_0\not \in \mathrm{H}^{r_0}(\mathbb{T})$. Thus, according to \cite[Lemma~3.2]{chousionis2014fractal} (or \cite[Theorem~III]{kapitanski1999does}), for almost all $t\in\mathbb{R}$, \[\operatorname{Re}\left(\mathrm{e}^{-i\partial^2 t} \Pi u_0\right)\not\in \bigcup_{r>r_0}\mathrm{H}^r(\mathbb{T}).\] Hence, by virtue of \eqref{withSzego}, we have 
\[
               u(\cdot,t)+\widehat{u_0}(0)= 2\operatorname{Re}\left(\mathrm{e}^{-i\partial^2 t} \Pi u_0\right)\not\in \bigcup_{r>r_0}\mathrm{H}^r(\mathbb{T}).
\]
\end{proof}

We now turn to the proof of Theorem~\ref{revfracBO}-(b). In this case, we cannot use  \eqref{uvrel} or \eqref{withSzego} to confirm the validity of the statement directly, because $u_0\in \mathrm{BV}(\mathbb{T})$ does not imply  $\Pi u_0\in \mathrm{BV}(\mathbb{T})$: indeed, $\operatorname{Im}(\Pi u_0)=\frac12{\mathcal H} u_0$ is unbounded as soon as $u_0$ has a jump discontinuity.
We adapt instead the ideas given in \cite{kapitanski1999does} and \cite{rodnianski2000fractal} for the analysis of the boundary-value problem \eqref{lBO}. 

For $\alpha\in\mathbb{R}$, the Besov spaces of order $\alpha$, denoted $\mathrm{B}_{p,\infty}^\alpha(\T)$, are defined as follows. 

Let $\chi:\mathbb{R}\longrightarrow [0,1]$ be a $\mathrm{C}^{\infty}$ function, such that 
\[\operatorname{supp} \chi=[2^{-1},2],\qquad \sum_{j=0}^\infty \chi (2^{-j}\xi)=1 \;\; \forall \xi\geq 1.\] 
Define $\chi_j$ by
\[
\chi_j(\xi)=\chi(2^{-j}\xi),\;\;j\in\mathbb{N}, \qquad \chi_0(\xi)=1-\sum_{j=1}^\infty \chi_j(\xi).
\]
The (Littlewood-Paley) projections of $f(x)=\sum_{n\in\mathbb{Z}}\hat{f}(n)\e^{inx}$,  function or distribution on $\mathbb{T}$, are given by
\[
       (K_jf)(x)=\sum_{n\in\mathbb{Z}}\chi_j(|n|)\hat{f}(n)\e^{inx}.
\]
Then $f\in \mathrm{B}_{p,\infty}^\alpha(\T)$ if and only if
\[
 \sup_{j=0,1,\ldots}
 2^{\alpha j}\|K_jf\|_{\mathrm{L}^{p}(\T)}<\infty.
\] 

We take $p=1$ in the proof of Corollary~\ref{lbodim} at the end of this section, but otherwise we will be concerned exclusively with the case $p=\infty$.

Below we use the following two properties of $\mathrm{B}_{\infty,\infty}^\alpha(\T)$:
 \begin{equation}\label{uuprime}
  f'\in  \mathrm{B}_{\infty,\infty}^\alpha(\T) \Longleftrightarrow f\in \mathrm{B}_{\infty,\infty}^{\alpha+1}(\T),\quad \forall \alpha\in\R;
 \end{equation}
\begin{equation}  \label{beho}
\mathrm{B}_{\infty,\infty}^\alpha(\T)=\mathrm{C}^\alpha(\T),\qquad \forall \alpha\in(0,1).
\end{equation}
The proof of these two statements is included in the appendix.
We will also make use of the next lemma, which is a consequence of \cite[Corollary~2.4]{kapitanski1999does}.

\begin{lemma}\label{goodt}
There exists a set $\mathcal K\subset \R$, whose complement $\mathcal K^c$ has measure zero but is dense in $\mathbb R$, such that the following holds true for all $t\in \mathcal{K}$. Given $\delta>0$, there exists a constant $C>0$ such that 
\begin{equation} \label{bareconditionq}
    \sup_{x\in \mathbb{T}}\left|\sum_{n=0}^{\infty}\chi_j(n)\e^{in^2t+inx}\right|\leq C 2^{\frac{j}{2}(1+\delta)},
\end{equation}
for all $j=0,1,\ldots$.
\end{lemma}
\begin{proof}
According to Dirichlet's Theorem, for every irrational number $a>0$ there are infinitely many positive integers $p,\,q\in\mathbb{N}$, such that $p$ and $q$ are co-prime, and
\begin{equation} \label{Dirichlet}
  \left| a - \frac{p}{q}\right|\leq \frac{1}{q^2}.
\end{equation}
By virtue of \cite[Lemma~4]{FJK1977}, there exists a constant $c_1>0$ such that, if the irreducible fraction $\frac{p}{q}$ satisfies \eqref{Dirichlet}, then
\begin{align*}
    \left|\sum_{n=M}^N \e^{2\pi i(an^2+bn)}\right| &=\left|\sum_{k=1}^{N-M} \e^{2\pi i(ak^2+bk)}\right|\\
&\leq c_1\left(\frac{N-M}{\sqrt{q}}+\sqrt{q}\right) \\
\end{align*}
for all $N\in\mathbb{N}$, $0< M<N$ and $b\in\mathbb{R}$. Here $c_1$ is independent of $a$ and $b$. Take any sequence $\{\omega_n\}$, such that $\omega_n=0$ for $n<M$ or $n>N$, and
\[
   \sum_{n=M}^N |\omega_{n+1}-\omega_n|\leq d.
\]
Since,
\begin{align*}
\left|\sum_{n=M}^N\omega_n\e^{2\pi i(an^2+bn)}\right|&=     \left| \sum_{n=M}^{N} (\omega_{n+1}-\omega_n)\sum_{k=M}^n\e^{2\pi i(ak^2+bk)}\right|\\
&\leq \sum_{n=M}^N|\omega_{n+1}-\omega_n|\left|\sum_{k=M}^n\e^{2\pi i(ak^2+bk)} \right|\\
&\leq d\sup_{n=M,\ldots,N}\left|\sum_{k=M}^n\e^{2\pi i(ak^2+bk)} \right|,
\end{align*}
then, 
\begin{equation} \label{cor2.4}
\left|\sum_{n=M}^N\omega_n \e^{2\pi i(an^2+bn)}\right|\leq dc_1 \left(\frac{N-M}{\sqrt{q}}+\sqrt{q}\right).
\end{equation}
This is \cite[Corollary~2.4]{kapitanski1999does}. 

Let $[a_0,a_1,\ldots]$ be the continued fraction expansion of the irrational number $a$,
\[
     a=a_0+\frac{1}{a_1+\frac{1}{a_2+\frac{1}{\cdots}}}.
\]
Then, the irreducible fractions, 
\[
     \frac{p_n}{q_n}=a_0+\frac{1}{a_1+\frac{1}{a_2+\cdots \frac{1}{a_{n-1}+\frac{1}{a_n}}}},
\]
are such that \eqref{Dirichlet} holds true with $\{p_n\}$ and $\{q_n\}$ increasing sequences. According to the Khinchin-L{\'e}vy Theorem, for almost every $a>0$ the denominators $q_n$ satisfy \cite[p.66]{Khinchin}
\[
     \lim_{n\to \infty} \frac{\log q_n}{n}=\rho,\qquad \rho=\frac{\pi^2}{12 \log 2}.
\]
If $a$ is such that this limit exists, then for all $j\in \mathbb{N}$ sufficiently large we can find quotients  $ \frac{p_n(j)}{q_n(j)}$  with denominators satisfying
$
     q_{n(j)}=2^{j(1+r_j)},
$
where $r_j\to 0$ as $j\to \infty$. Indeed, we can take $n(j)$ equal to the integer part of $j(\log 2)/\rho$\footnote{Note that $\{q_{n(j)}\}$ is not a subsequence of $\{q_n\}$, since $\log2/\rho<1$ and therefore indices may be repeated.}. This choice implies
$\displaystyle{
   \lim_{j\to \infty}\frac{\log q_{n(j)}}{j}=\log 2.
}$

Let $\mathcal{K}$ be the  set of positive times of the form $t=2\pi a$ such that the sequence of quotients of $a$ satisfies the conditions of the previous paragraph.  Let $t\in  \mathcal{K}$ and fix $\delta>0$. 
Let $J>0$ be such that $|r_j|<\delta$ for all $j\geq J$.
Taking $M=2^{j-1}$, $N=2^{j+1}$, $\omega_n=\chi_j(n)$, and
\[
     d=2\sup_{\xi\in \R}|\chi'(\xi)|,
\]
in \eqref{cor2.4}, yields
\begin{align*}
     \sup_{x\in \T} \left| \sum_{n=0}^{\infty}\chi_{j}(n)\e^{in^2t+inx}\right|&=\sup_{x\in \T} \left| \sum_{n=2^{j-1}}^{2^{j+1}}\chi_{j}(n)\e^{in^2t+inx}\right|\\ 
&\leq dc_1\left(\frac{2^{j+1}-2^{j-1}}{\sqrt{q_{n_j}}}+\sqrt{q_{n_j}}\right)\\
&\leq dc_1 \left(\frac{2^{j-1}3}{2^{\frac{j}{2}(1-\delta)}}+2^{\frac{j}{2}(1+\delta)}\right)
\\ &\leq c_2 2^{\frac{j}{2}(1+\delta)},
\end{align*}
for all $j\geq J$.
This implies \eqref{bareconditionq} for sufficiently large $C>0$.
\end{proof}

\begin{proof}[Proof of Theorem~\ref{revfracBO}-(b)]
Let $u_0\in \operatorname{BV}(\T)$. 
Define the periodic distribution,
\[
       E_t(x)=\sum_{n=1}^{\infty}\left[\e^{inx+in^2t}+\e^{-inx-in^2t}\right].
\]
Since the Fourier coefficients $\e^{i|n|nt}$ of $E_t$ are unimodular,  the series converges in the weak sense of distributions and determines $E_t$ uniquely for all $t\in \R$, \cite[Theorems 11.6-1 and 11.6-2]{Zemanian1965}.
Moreover, for all $t\in\mathbb{R}$, $E_t\in \mathrm{B}^{\beta}_{\infty, \infty}(\T)$ for all $\beta<-1$. 

Let $t\in \mathcal{K}$, with $\mathcal{K}\subset\mathbb{R}$ as in Lemma~\ref{goodt}. Then it follows from \eqref{bareconditionq} that in fact the stronger inclusion $E_t\in\mathrm{B}_{\infty,\infty}^\beta(\T)$ for all $\beta<-\frac12$.  Define the periodic distribution $H_t$ by $H'_t=E_t$, namely 
$$
H(t)=\sum_{n=1}^{\infty}\frac{\e^{inx+in^2t}-\e^{-inx-in^2t}}{in}=\sum_{n\neq 0,\,n=-\infty}^\infty\frac{\e^{inx+in|n|t}}{in}.
$$ 
Then, according to \eqref{uuprime} and \eqref{beho},  $H_t\in\mathrm{C}^{\beta+1}(\T)$ for all $\beta<-\frac12$. 

Note that
$$
\widehat {u_0}(n)=\frac{1}{2\pi in} \int_{\mathbb T} \e^{-iny}\,\mathrm{d}u_0(y)=\frac{\widehat{\mu_0}(n)}{in},\quad n\neq 0,
$$
where $\mu_0$ is the Lebesgue-Stieltjes measure associated to $u_0$, which satisfies $|\mu_0|(\T)<\infty$. 
Then the solution of \eqref{lBO}, given by \eqref{solution},  can be expressed in terms of $H_t$ as follows:
\begin{align*}
 u(x,t)&= \widehat {u_0}(0)+\sum_{n\neq 0,\,n=-\infty}^\infty \e^{i|n|n t}\widehat{u_0}(n)\e^{inx} \\ 
&= \widehat {u_0}(0)+\sum_{n\neq 0,\,n=-\infty}^\infty \frac{\e^{i|n|n t}\widehat{\mu_0}(n)}{in}\e^{inx} \\
&= \widehat {u_0}(0)+(H_t * \mu_0)(x).
\end{align*}
Hence, since $\mu_0$ is a bounded measure, we indeed have $u(\cdot,t)\in \mathrm{C}^{\alpha}(\T)$ for all $\alpha<\frac12$. 
\end{proof}
We conclude this section giving the proof of Corollary~\ref{lbodim}.

\begin{proof}[Proof of Corollary~\ref{lbodim}]
Let $D$ denote the upper Minkowski dimension of the graph of $u(\cdot,t)$. By virtue of Theorem~\ref{revfracBO}-(b), it follows that $D\leq \frac32$ for almost all $t\in\mathbb{R}$, see \cite[Corollary~11.2]{Falconer1990}. Moreover, since \[\mathrm{B}^{r_1}_{1,\infty}(\mathbb{T})\cap \mathrm{B}^{r_2}_{\infty,\infty}(\mathbb{T})\subset \mathrm{H}^r(\mathbb{T})\]for all $r<\frac{r_1+r_2}{2}$, then $u(\cdot,t)\not\in \mathrm{B}^{r_1}_{1,\infty}(\mathbb{T})$ whenever $r_1>\frac12$, for those $t$ for which the conclusions of Theorem~\ref{revfracBO}-(b) and (c) hold. Thus, by virtue of \cite[Theorem~4.2]{DeliuJawerth1992}, we also have the complementary bound $D\geq \frac32$ for all such $t$. This ensures the claim made in the corollary. 
\end{proof}

\section{Illustration of the main results}\label{s3}

In this final section we examine the claims of Theorem~\ref{revfracBO} and Corollary~\ref{lbodim}. We illustrate their meaning for the case that the initial condition is a step function and for specific values of the time variable.

We first consider how Theorem~\ref{revfracBO}-(a) implies the cusp revival phenomenon.
Assume that $u_0(x)=\mathbbm{1}_{[-\frac{\pi}{2},\frac{\pi}{2}]}(x)$. Then, according to the statement of the theorem, for $t\in 2\pi\mathbb{Q}$ the solution of \eqref{lBO} is the summation of two functions; a finite linear combination of characteristic functions (a simple function) and its Hilbert transform. Since the periodic Hilbert transform of the characteristic function of a single interval can be computed explicitly as  \begin{equation} \label{logsing}
    \mathcal{H} \mathbbm{1}_{[a,b]}(x)=\frac{1}{\pi}\log\left|\frac{\sin\Big(\frac{x-a}{2}\Big)}{\sin\Big(\frac{x-b}{2}\Big)}\right|,
\end{equation} for $a,\,b\in \mathbb{T}$ with $-\pi\leq a<b<\pi$,
it follows by linearity that the graph of the solution displays finitely many logarithmic cusps for any $t\in 2\pi\mathbb{Q}$.  This is illustrated in Figure~\ref{Fig1}.

\begin{figure}
\includegraphics[width=\textwidth]{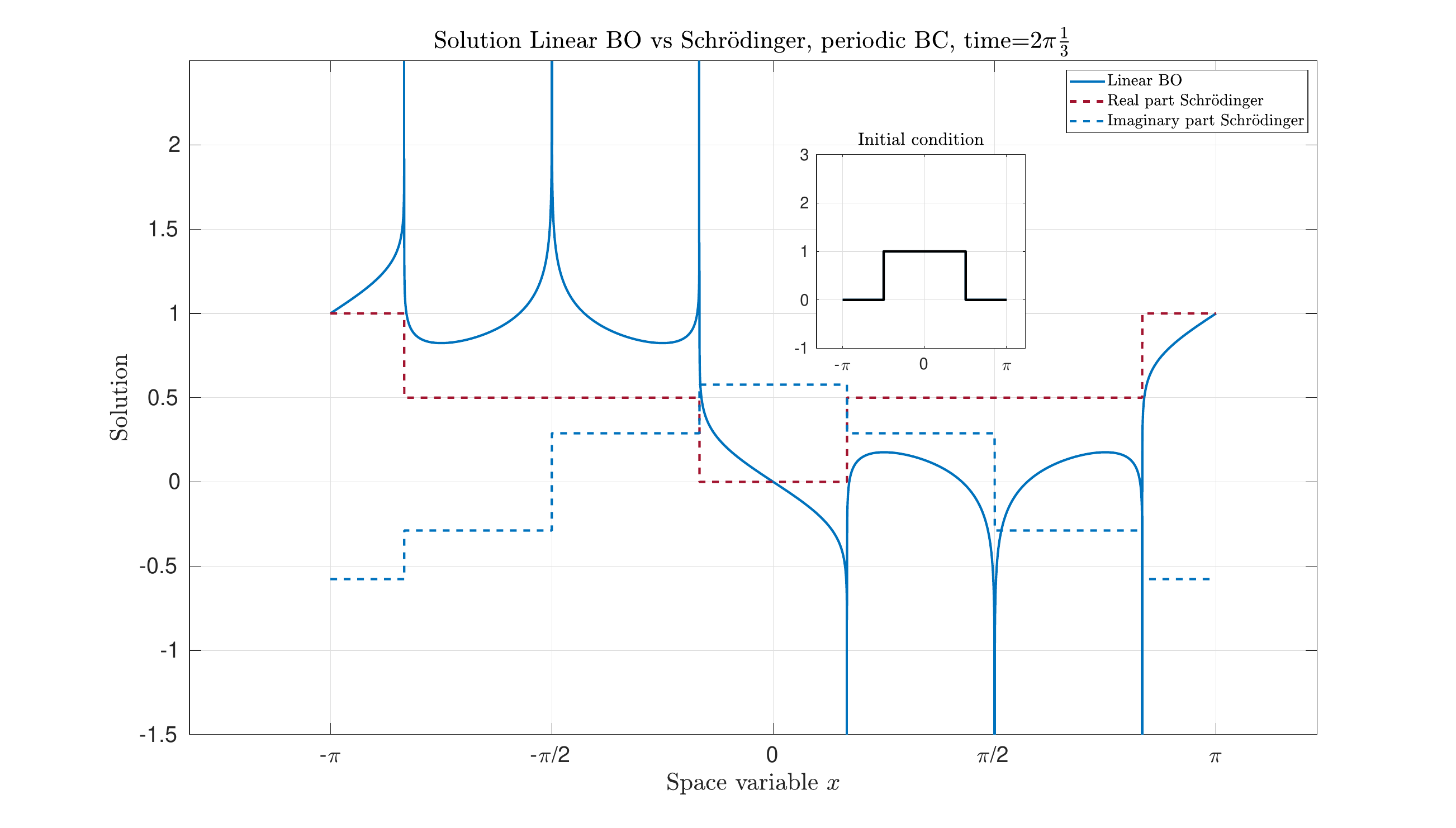}
\caption{Solution of \eqref{lBO} for $u_0(x)=\mathbbm{1}_{[-\frac{\pi}{2},\frac{\pi}{2}]}(x)$ at time $t=2\pi\frac{1}{3}$ superimposed on the real and imaginary parts of the solution of \eqref{lsc}. 
The cusp singularities in the solution of \eqref{lBO} 
correspond to jump singularities in either part of the solution of \eqref{lsc}. }
\label{Fig1}
\end{figure}

\begin{figure}
\includegraphics[width=\textwidth]{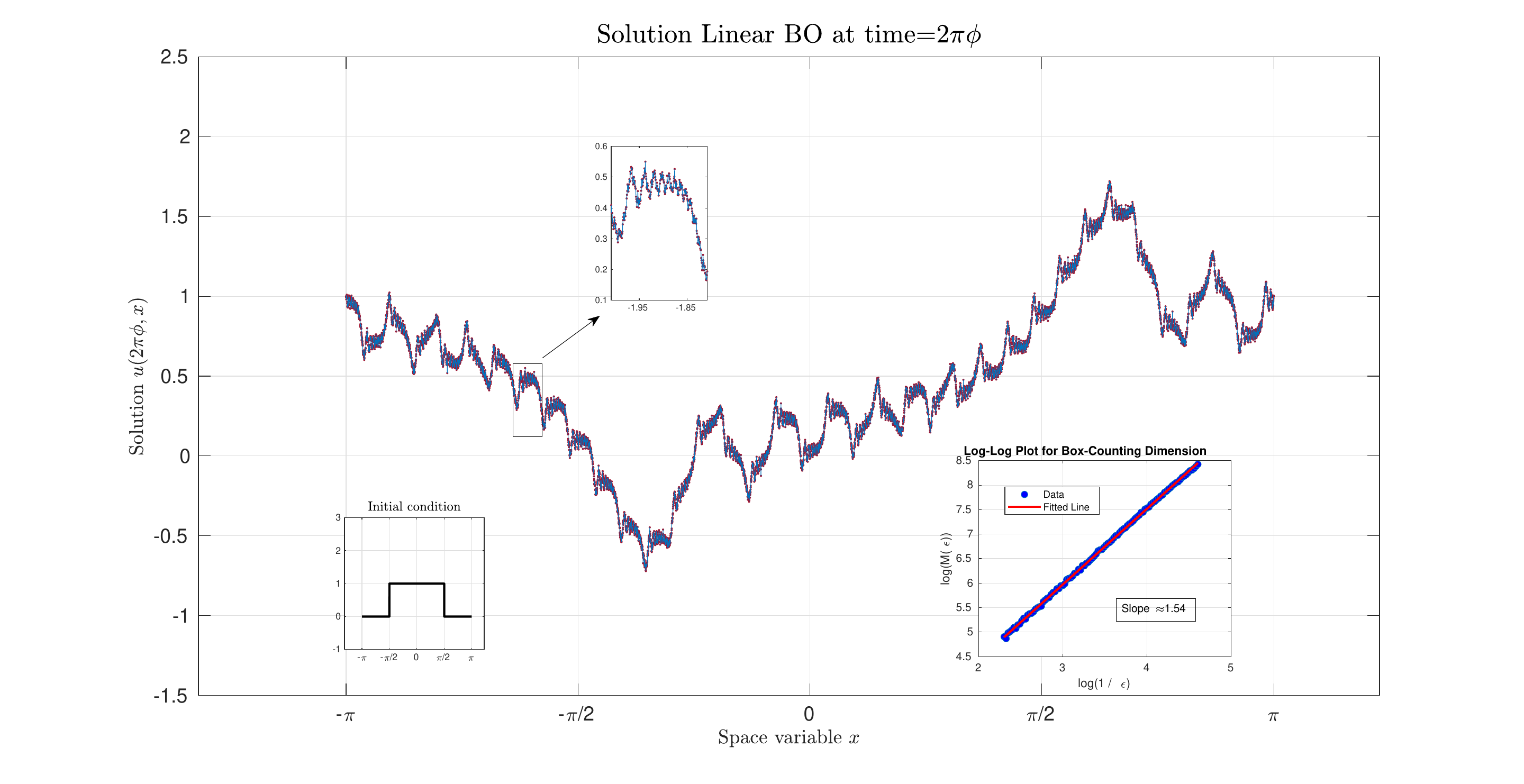}
\caption{Solution for $\frac{t}{2\pi}$ a rational approximation of $\phi\sim\frac{p}{q}$ for $p=F_{16}=2584$ and $q=F_{15}=1597$. Note that $|\phi-\frac{p}{q}|<1.7\times 10^{-6}$. The estimate of the box counting dimension is $D=1.54$. \label{figgr}}	
\end{figure}

\begin{figure}
\includegraphics[width=\textwidth]{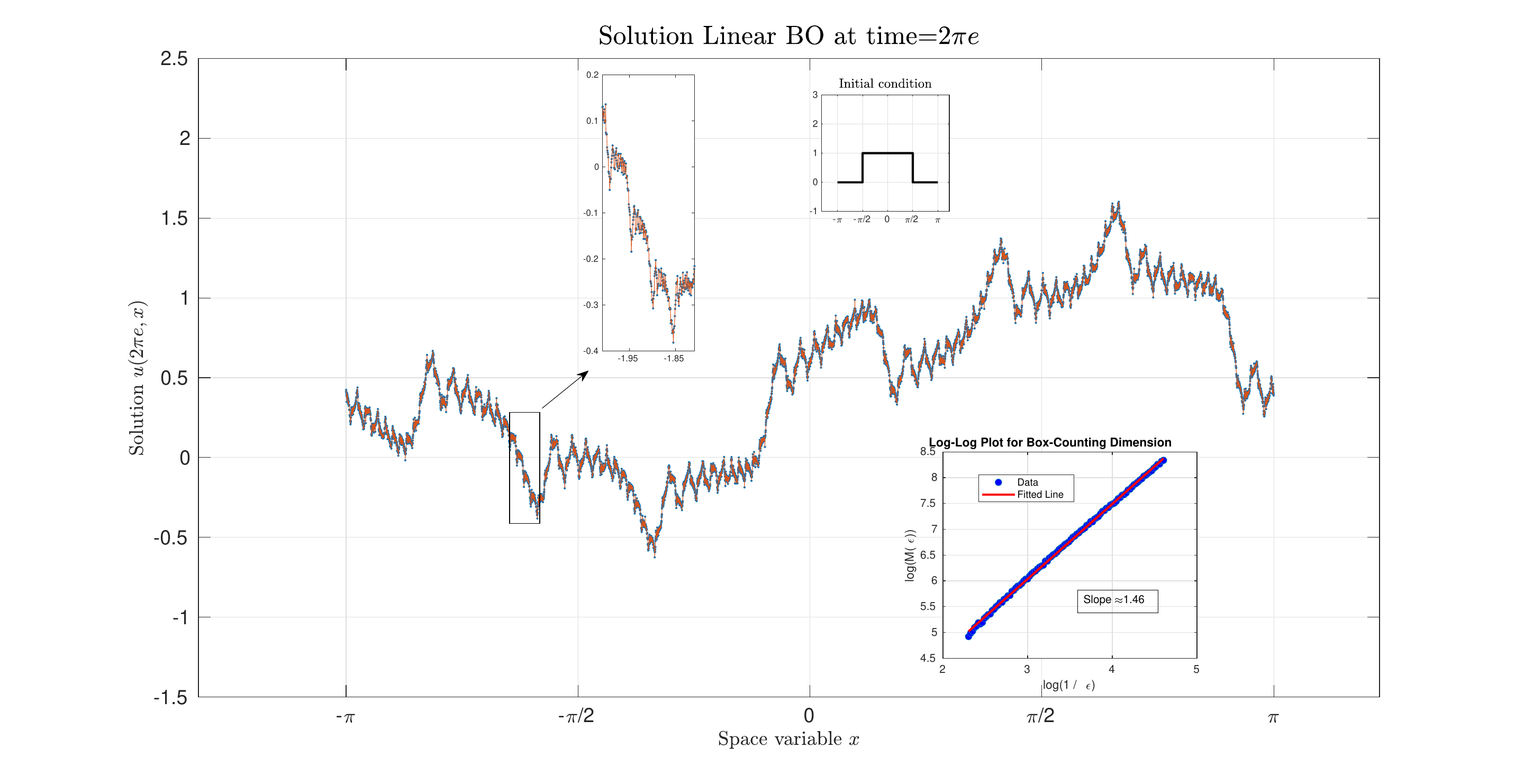}
\caption{Solution for for $\frac{t}{2\pi}$ a rational approximation of $\e\sim\frac{p}{q}$ for $p=23225$ and $q=8544$. Note that $|\e-\frac{p}{q}|<6.7\times 10^{-9}$.  The estimate of the box counting dimension is $D=1.46$.\label{figex}}
\end{figure}

\medskip

We now examine the result of Corollary~\ref{lbodim} and confirm the conclusion that the fractal dimension of the graph of the solution is equal to $\frac32$ for specific irrational times.   We consider rational approximations to two different values of the time variable:  $t=2\pi \phi$, where $\phi=\frac{1+\sqrt{5}}{2}$ is the Golden Ratio, and  $t=2\pi \e$. Both satisfy \eqref{bareconditionq}.

As the denominator $q$ in Theorem~\ref{revfracBO}-(a) increases, the number of singularities of the solution increases. In the limit as $\frac{p}{q}$ approaches almost every irrational number, Theorem~\ref{revfracBO}-(b) implies that the solution will approach a continuous function. In figures~\ref{figgr} and \ref{figex} we show an approximation of the values of the solution for 10k points uniformly distributed on $x\in(-\pi,\pi]$. Note that, for the  initial profile $u_0=\mathbbm{1}_{[-\frac{\pi}{2},\frac{\pi}{2}]}$, shown in the figures, we have \eqref{logsing} for $a=-\frac{\pi}{2}$ and $b=\frac{\pi}{2}$. We generate the values of the solution by using directly the formula \eqref{bosum} evaluated at the 10k nodes partitioning the segment $(-\pi,\pi]$.  

Embedded in figures~\ref{figgr} and \ref{figex}, we give an estimate of the box-counting dimension $D$ of each graph. This, in turns, is an upper bound for the upper Minkowski dimension. Note that both numbers are close to the value $\frac32$, namely, $D=1.54$ and $1.46$, respectively. To arrive at these prediction, we have implemented the Algorithm~\ref{algobox}, as follows. We compute the number $M(\epsilon)$ of square boxes of side $\epsilon$ covering the graph of the solution, for a range of $\epsilon$ as shown in the graphs, then interpolate the approximated box-counting dimension
\[
    D=\lim_{\epsilon \to 0}\frac{\log M(\epsilon)}{\log \frac{1}{\epsilon}}
\]  
using the slope of the linear fitting.

Both graphs point to the fractal nature of the solution. Indeed, they seem to indicate a self-similar pattern in the solution as $x$ increases and also they appear to be nowhere differentiable curves.

\begin{algorithm} \caption{Function for counting the number of boxes of side $\epsilon$ required to cover the graph interpolated by the data $A = [x, y]$, where $x$ and $y$ are vectors of size $N$.
\label{algobox}} 
\begin{algorithmic}[1] 
\Procedure{numbox}{$x,\,y,\,\epsilon$}       
\Comment{$\epsilon=$\,size of the boxes}   
\State $N=\operatorname{size}(x)$;   
\State $a=\min(x)$; 
\State $b=\max(x)$;    
\If{$\epsilon<2(b-a)/N$}    
       \textbf{break}  \Comment{Break for $<2$ pts per $x$ coord}
  \EndIf 
       \State $C=0$;    \Comment{Initialise count variable}
\State $N_{\mathrm{towers}}=\operatorname{floor}((b-a)/\epsilon)+1$;\Comment{Number of towers} 
\For{$k=1:N_{\mathrm{towers}}$} \Comment{Loop on each tower}
 \State $I=\operatorname{find}\big\{j:(k-1)\epsilon+a\leq x(j)<k\epsilon+a\big\}$;
 \State $J=\max(y(I))-\min(y(I))$;
 \State $N_{\mathrm{boxes in tower}}=\operatorname{floor}(J/\epsilon)+1$; \Comment{Count boxes in each tower}
  \State $C=C+N_{\mathrm{boxes in tower}}$; \Comment{Add to total count}
\EndFor
\State    \Return $C$ \Comment{After loop return total}
  \EndProcedure  \end{algorithmic} \end{algorithm}

\section*{Acknowledgements}
We kindly thank O.~Pocovnicu for her useful comments during the preparation of this manuscript. LB was supported by COST Action CA18232. BM has been supported by EPSRC through a MAC-MIGS summer studentship. BP was partially supported by a Leverhulme Research Fellowship.

\bibliographystyle{amsplain}

\appendix

\section{The Besov spaces $\mathrm{B}^\alpha_{\infty,\infty}(\T)$}
Consider the definition of $B^{\alpha}_{\infty, \infty}(\T)$ given above. Let $g\in\mathcal{S}(\R)$ be such that $\mathcal{F}g(\xi)=\chi(\xi)$, where 
\[
    \mathcal{F}f(\xi)=\int_{\R}f(x)\e^{-i\xi x}\,\mathrm{d}x
\]
is the Fourier transform. Then,
$
     (\mathcal{F}g_j)(\xi)=\chi_{j}(\xi)
$ for $g_j(x)=2^{j}g(2^j x)$. 

If $f\in\mathcal{S}(\R)$, Poisson's Summation Formula prescribes that,
\[
    \sum_{n\in \Z}f(x+n)=\sum_{n\in \Z}(\mathcal{F}f)(2\pi n)\e^{2\pi inx}
\]
for all $x\in \R$. Letting $\tilde{f}(x)=f(2\pi x)$, gives
\[
     (\mathcal{F}\tilde{f})(\xi)=\frac{1}{2\pi} (\mathcal{F}f)\left(\frac{\xi}{2\pi}\right).
\]
Then,
\[
     \sum_{k\in \Z}f(z+2\pi k)=\frac{1}{2\pi} \sum_{k\in \Z} (\mathcal{F}f)(k)\e^{inz}.
\]
Hence, we can represent the projections $K_j$ of any periodic distribution $F$, as
\begin{align*}
(K_jF)(x)&=\sum_{k=-\infty}^{\infty} \chi_j(|k|)\left(\frac{1}{2\pi}\int_{\T}F(y)\e^{-iky}\,\mathrm{d}y\right)\e^{ikx} \\
&=\int_{\T}\left(\frac{1}{2\pi} \sum_{k=-\infty}^{\infty} \chi_j(|k|)\e^{ik(x-y)}\right)F(y)\,\mathrm{d}y \\&=\int_{\T}\left(\sum_{k=-\infty}^{\infty} g_j(x-y+2k\pi)\right)F(y)\,\mathrm{d}y\\&=\sum_{k=-\infty}^{\infty} \int_{\T} g_j(x-y+2k\pi)F(y-2k\pi)\,\mathrm{d}y \\
&=\int_{\R}g_j(x-y)F(y)\,\mathrm{d}y=(g_j\star F)(x).
\end{align*}
for all $x\in\R$. Here the symbol ``$\star$'' denotes the convolution on $\R$.

Now, according to \cite[Lemma~2.1, p.52]{bcd} in the case $p=\infty$, there exists a constant $C>0$ which only depends on $r_1$, $r_2$ and $\lambda$, ensuring the following estimates. For any function $u\in\mathrm{L}^\infty(\R)$, such that 
\[
    \operatorname{supp} (\mathcal{F}u)\subset \lambda \{\xi\in\R\,:\, 0<r_1\leq |\xi|\leq r_2\},
\]
we have
\begin{equation}\label{Bernstein}
\frac{\lambda}{C}\|u\|_{\mathrm{L}^{\infty}(\R)}\leq \|u'\|_{\mathrm{L}^{\infty}(\R)}
\leq C\lambda \|u\|_{\mathrm{L}^{\infty}(\R)}.
\end{equation}
This is some times called Bernstein's Inequality.

\subsubsection*{Proof of \eqref{uuprime}}
Take $u=g_j\star F$, $\lambda=2^j$, $r_1=2^{-1}$ and $r_2=2$ in \eqref{Bernstein}. Then, the left hand side inequality yields,
\[
    2^{(\alpha+1)j}\|K_jF\|_{\mathrm{L}^{\infty}(\T)}\leq C
     2^{\alpha j}\|K_j(F')\|_{\mathrm{L}^{\infty}(\T)}<\infty,
\]
for $F'\in B^{\alpha}_{\infty, \infty}(\T)$. Conversely, the right hand side inequality yields,
\[
     2^{\alpha j}\|K_j(F')\|_{\mathrm{L}^{\infty}(\T)}\leq C2^{(\alpha+1)j}\|K_jF\|_{\mathrm{L}^{\infty}(\T)}<\infty,
\]
for $F\in B^{\alpha+1}_{\infty, \infty}(\T)$. 

\subsubsection*{Proof of \eqref{beho}}
We know that $f\in \mathrm{C}^{\alpha}(\T)$, if and only if $S_1+S_2<\infty$, for
\[
    S_1=\sup_{x\in\T}|f(x)|
\]
and
\[
S_2=\sup_{\substack{x\in \T \\ h\not=0}}\frac{|f(x+h)-f(x)|}{|h|^{\alpha}}.
\]
Recall that, $f\in\mathrm{B}^\alpha_{\infty,\infty}(\T)$, if and only if $R<\infty$, for
\[
     R=\sup_{j=0,1,\ldots}\sup_{x\in\T}2^{\alpha j}|K_jf(x)|.
\]

Let $f\in \mathrm{B}_{\infty,\infty}^{\alpha}(\T)$. We show that $S_1$ and $S_2$ are finite.
Firstly note that,
\[
     f(x)=\sum_{j=0}^\infty K_jf(x).
\]
Hence, 
\[
     S_1\leq \sum_{j=0}^\infty \|K_jf\|_{\mathrm{L}^\infty(\T)}\leq \sum_{j=0}^\infty
\frac{R}{2^{\alpha j}}<\infty.
\]
Here we have used that $\alpha>0$.

Now, if
\[
    S_3=\limsup_{h\to 0} \left(\sup_{x\in \T}\frac{|f(x+h)-f(x)|}{|h|^{\alpha}}\right)<\infty,
\]
then $S_2<\infty$. For $j=0,1,\ldots$, let
\[
    S_4(j)=\limsup_{h\to 0} \left(\sup_{x\in \T}\frac{|K_j(f(x+h)-f(x))|}{|h|^{\alpha}}\right).
\]
Then, on the one hand,
\[
    S_3\leq \sum_{j=0}^\infty S_4(j).
\]
On the other hand, by the Mean Value Theorem, for suitable $|h_j|<2^{-2j}$,
\begin{align*}
S_4(j)&\leq \sup_{\substack{x\in \T \\ 0<|h|\leq 2^{-2j}}}\frac{|K_jf(x+h)-K_jf(x)|}{|h|^{\alpha}} \\
&\leq \sup_{\substack{x\in \T \\ 0<|h|\leq 2^{-2j}}}\frac{|(K_jf)'(x+h_j)||h|}{|h|^{\alpha}} \\
&= \sup_{0<|h|\leq 2^{-2j}}|h|^{1-\alpha}\sup_{x\in \T}|(g_j\star f)'(x+h_j)| \\
&\leq 2^{-2j(1-\alpha)}\|(g_j\star f)'\|_{\mathrm{L}^\infty(\R)}\\
&\leq C 2^j2^{-2j(1-\alpha)} \|g_j\star f\|_{\mathrm{L}^\infty(\R)}\\
&= C2^{-j(1-\alpha)}2^{\alpha j}\|K_jf\|_{\mathrm{L}^\infty(\T)}\\
&\leq C R 2^{-j(1-\alpha)}.
\end{align*}
Thus, indeed, $S_3<\infty$. Here we have used that $1-\alpha>0$. This confirms that $\mathrm{B}_{\infty,\infty}^\alpha(\T)\subseteq \mathrm{C}^{\alpha}(\T)$.

\medskip

Now, let us show that $\mathrm{C}^{\alpha}(\T)\subseteq\mathrm{B}_{\infty,\infty}^\alpha(\T)$. Assume that $f\in \mathrm{C}^\alpha(\T)$. That is $S_1<\infty$ and $S_2<\infty$. Considering $f$ as a periodic function of $x\in \R$, we have
\[
    S_1=\sup_{x\in\R}|f(x)|<\infty
\]
and
\[
S_2=\sup_{\substack{x\in \R \\ h\not=0}}\frac{|f(x+h)-f(x)|}{|h|^{\alpha}}<\infty.
\]
Our goal is to show that $R<\infty$.

Since $g\in\mathcal{S}(\R)$, then there exists a constant $c_3>0$, such that
\[
    |g_j(x)|\leq c_3\frac{2^j}{(1+2^j|x|)^{2}},
\]
for all $x\in \R$. 
Now for any $\varphi\in \R$, thought of as a constant periodic function, we have that $(g_j\star \varphi)(x)=\varphi\chi_{_j}(0)=0$ for $j=1,2,\ldots$. Then,
\[
    (g_j\star f)(x)=(g_j\star (f+\varphi))(x)
\]
for all $x\in \R$ and $j\in\N$. Thus,
\begin{align*}
|(g_j\star f)(x)|&\leq \int_{\R} |g_j(y)||f(x-y)+\varphi|\,\mathrm{d}y \\
&\leq c_3 2^j \int_{\R}\frac{|f(x-y)+\varphi|}{(1+2^j|y|)^2}\,\mathrm{d}y \\
&=c_3 \int_{\R}\frac{\left|f\left(x-\frac{z}{2^j}\right)+\varphi\right|}{(1+|z|)^2}\,\mathrm{d}z.
\end{align*}
for all $x\in \R$, $\varphi\in \R$ and $j\in \N$. 

This gives, taking $\varphi=-f(x)$, that
\[
     2^{\alpha j}|(g_j\star f)(x)|\leq c_3 2^{\alpha j}\int_{\R}\frac{\left|f\left(x-\frac{z}{2^j}\right)-f(x) \right|}{(1+|z|)^2}\,\mathrm{d}z=A_j(x)+B_j(x),
\]
where we split the integral as follows. The first term is,
\begin{align*}
    A_j(x)&=c_3 2^{\alpha j} \int_{-2^j}^{2^j}\frac{\left|f\left(x-\frac{z}{2^j}\right)-f(x)\right|}{(1+|z|)^2}\,\mathrm{d}z \\
&=c_3 \int_{-2^j}^{2^j}\frac{|z|^{\alpha}\left|f\left(x-\frac{z}{2^j}\right)-f(x)\right|}{\left(\frac{|z|}{2^j}\right)^{\alpha}(1+|z|)^2}\,\mathrm{d}z \\
&\leq c_3 S_2\int_{-\infty}^{\infty} \frac{|z|^\alpha}{(1+|z|)^2}\,\mathrm{d}z\leq c_4<\infty,
\end{align*}
for all $j=1,2,\ldots$ and $x\in \T$. Here we have used that $0<\alpha<1$. The second term is, 
\begin{align*}
B_j(x)&=c_3 2^{\alpha j} \int_{|z|\geq 2^j}\frac{\left|f\left(x-\frac{z}{2^j}\right)-f(x)\right|}{(1+|z|)^2}\,\mathrm{d}z \\
&\leq c_3 2^{\alpha j}2 S_1 \int_{|z|\geq 2^j} \frac{\mathrm{d} z}{(1+|z|)^2} \\
&\leq c_5 S_1 2^{(\alpha-1)j}\leq c_6<\infty,
\end{align*}
for all $j=1,2,\ldots$ and $x\in \T$. Here we have used that $\alpha<1$. Then $R\leq c_4+c_6<\infty$. This completes the proof of \eqref{beho}.

\end{document}